\documentclass{amsart}

\usepackage{color}
\usepackage{amsmath}
\usepackage{amsfonts}
\usepackage{amssymb,enumerate}
\usepackage{enumitem}
\usepackage{amsthm}
\usepackage{tikz}

\usetikzlibrary{matrix}
\usepackage{csquotes}
\usetikzlibrary{positioning}
\usepackage{comment}
\usepackage[all]{xy}
\usepackage[notref,notcite]{}
\usepackage[plainpages=false,pdfpagelabels]{hyperref}
\hypersetup{
    colorlinks=false,
    pdfborder={0 0 0}
}

\usepackage[utf8]{inputenc}
\usepackage[noadjust]{cite}
\usepackage{lineno}

\theoremstyle{definition}
\newtheorem{lem}{Lemma}[section]
\newtheorem{cor}[lem]{Corollary}

\newtheorem{prop}[lem]{Proposition}
\newtheorem{thm}[lem]{Theorem}

\newtheorem{conj}[lem]{Conjecture}
\newtheorem{defn}[lem]{Definition}
\newtheorem{ex}[lem]{Example}

\newtheorem*{equalnorm}{Equal Norm Problem}
\newtheorem*{partition}{Partition Problem}

\newcommand{\indicator}{\text{\usefont{U}{bbold}{m}{n}1}}

\newcommand{\spec}{\operatorname{Spec}}

\newcommand{\OK}{\mathcal{O}_K}
\newcommand{\Cl}{\textup{Cl}}

\newcommand{\ideal}[1]{\mathfrak{#1}}

\newcommand{\p}{\ideal{p}}
\newcommand{\q}{\ideal{q}}

\newcommand{\bbz}{\mathbb{Z}}
\newcommand{\bbn}{\mathbb{N}}
\newcommand{\bbq}{\mathbb{Q}}

\newcommand{\imp}{\ensuremath{\Rightarrow{}}}

\newcommand{\mon}{D(\mathfrak{B}_\Gamma(Cl_K))}

\renewcommand{\geq}{\geqslant}
\renewcommand{\leq}{\leqslant}

\numberwithin{equation}{section}

\begin{document}
\email[J.~Coykendall]{jcoyken@clemson.edu}
\email[J.~Kettinger]{jkettin@clemson.edu}
\author{Jim Coykendall and Jared Kettinger}
\address{Blank}

\title{Galois Action and Localization in Number Fields}
\begin{abstract}
    For a Galois number field $K$, the Galois group $\text{Gal}(K/\bbq)$ acts on the class group $\Cl_K$ in a very natural way: $\sigma\cdot[I]=[\sigma(I)]$ for any $\sigma \in \text{Gal}(K/\bbq)$, $[I]\in \Cl_K$. In this paper, we will explore how the unique properties of this group action work together to elucidate the relationship between these two groups---developing and expanding upon some known results from a new perspective. To this end, we explore the class groups of localizations of the ring of integers $\OK$. These turn out to be powerful tools for understanding $\Cl_K$ and overrings of $\OK$. The paper concludes with some interesting observations about normset arithmetic and complexity---topics intimately related to this action. 
\end{abstract}

\maketitle
\sloppy

\section{Introduction and Notation}

Throughout this paper, $K$ will denote a Galois number field with Galois group $G:= \text{Gal}(K/\bbq)$, and $\OK$ its ring of integers with class group $\Cl_K$ and class number $h_K = |\Cl_K|$. The Galois group of $K$ acts on the class group in a very natural way. For any $\sigma \in G$ and $[I] \in \Cl_K$, we can define the action $\sigma\cdot [I] = [\sigma (I)]$. This group action provides us with various tools for characterizing the relationship between the two groups. More than this, we also have the additional property that $\sigma\cdot ([I][J]) = (\sigma \cdot [I])(\sigma\cdot[J])$ for any $\sigma \in G$ and $[I],[J] \in \Cl_K$. This induces a map from $G$ to $\text{Aut}(\Cl_K)$ given by $\sigma \overset{\psi}{\mapsto} \phi$ where $\phi([I]) = \sigma \cdot [I] = [\sigma(I)]$. Finally, we have the \enquote{norm property} that $\prod_{\sigma \in G}\sigma \cdot [I] =  \left[\prod_{\sigma \in G}\sigma(I)\right] =  \text{Prin}(\OK)$ for any $[I] \in \Cl_K$---for details, see \cite{Marcus2018}. Note, in this paper we will study the case where $K$ is Galois over $\bbq$ with the knowledge that many of the theorems herein apply also to the case when $K/L$ is a Galois extension of number fields with $\mathcal{O}_L$ a principal ideal domain (PID). Now, these characteristics of the action motivate the following definition. 

\begin{defn}\label{maindef}
    Let \( G \) and \( A \) be groups with \( A \) abelian, and let
    \[
        \alpha: G \times A \longrightarrow A, \quad (g, a) \mapsto g \cdot a
    \]
    be a function. If \( \alpha \) satisfies the following properties:
    \begin{enumerate}
        \item \( g_1 \cdot (g_2 \cdot a) = (g_1 g_2) \cdot a \) for all \( g_1, g_2 \in G \), \( a \in A \),
        \item \( e_G \cdot a = a \) for all \( a \in A \),
        \item \( g \cdot (a_1 a_2) = (g \cdot a_1)(g \cdot a_2) \) for all \( g \in G \), \( a_1, a_2 \in A \),
        \item \( \displaystyle\prod_{g \in G} (g \cdot a) = e_A \) for all \( a \in A \),
    \end{enumerate}
    then we say that \( \alpha \) is a \textit{norm-like action}.
\end{defn}

Throughout this paper, we will see how properties 1-4 can be used to place strong restrictions on the structure of $\Cl_K$ given $G$ and vice versa. We will place a particular emphasis on the techniques developed and some applications to factorization theory. While the class group is inherently connected to the arithmetic of rings of integers, this particular action is intimately related to the factorization properties of the normset of a ring of integers. This connection was first studied by the first author in \cite{C96b}. The recent work of Geroldinger, Halter-Koch, and Zhong in \cite{GHKZ22} has inspired renewed interest in this action as it relates to normset factorization---particularly in the quadratic case. 


Throughout this paper, we will present new proofs of some existing theorems as well as extend various results on the structure of $\Cl_K$ using new techniques. Previous work done on this subject has been primarily through the lens of representation theory---considering $\Cl_K$ as a $G$-module. This approach was utilized by Fr{\"o}hlich in 1952 (\cite{Frolich1952}) and subsequently many others. Cornell and Rosen used the Galois module structure of the class group in \cite{cornell1981} to give group-theoretic constraints on the structure of $\Cl_K$ when $h_K$ is known. We will begin with a similar approach and introduce new techniques to expand upon their results. Others such as Lemmermeyer  (\cite{Lemmermeyer2003}) and Iwasawa (\cite{Iwasawa1966}) used the Galois module structure of $\Cl_K$ along with assumptions on intermediate fields to give various results on the $p$-rank of the class group. In \cite{komatsu2001galois}, Komatsu and Nakano generalize this work, using the same structure and some group-theoretic considerations to give further restrictions on the $p$-rank of $\Cl_K$---once again using assumptions on the intermediate fields. We will develop some of the tools used throughout these works in this chapter from a strictly group-theoretic point of view. However, we diverge from these methods by considering the class groups of overrings of $\mathcal{O}_K$ rather than $\Cl_F$ for number fields $F \subseteq K$. This allows us to derive related results under fewer assumptions. For a thorough review of the previous work done in this area, see \cite{Metsankyla2007}.

In the following section, we explore how conditions 1 and 2, the typical conditions of a group action, together with 4 place restrictions of the structure of $\Cl_K$ for a Galois number field $K$ of given degree. In Section 3, we focus on applications of condition 3 to the inverse class group problem. In particular, we ask which number fields $K$ can have a given abelian group as their class group. Section 4 considers the structure of the class group of localizations of the form $\OK[\frac{1}{\alpha}]$. This ultimately strengthens many of the previously developed results by, among other things, allowing us to restrict to Sylow $p$-subgroups of the class group. In Section 5, the authors consider the class groups of overrings of Dedekind domains broadly. The paper culminates with a characterization of those overrings of Galois rings of integers whose class groups admit the same norm-like group action described above. This ultimately demonstrates that all class groups of interest can be realized as the class group of a simple localization. Finally, Section 6 considers the arithmetic of norms of algebraic rings of integers---a topic closely tied to this action. We conclude by drawing an interesting connection between this arithmetic and the partition problem.

\section{Direct Consequences of the Group Action}

In this section, we produce results on the structure of the class group of Galois number fields which follow directly from the action of the Galois group $G = \text{Gal}(K/\bbq)$ on the class group $\Cl_K$ in conjunction with the norm property $\prod_{\sigma\in G} \sigma(I) \in \text{Prin}(\mathcal{O}_K)$ for any ideal $I \subseteq \mathcal{O}_K$. We begin with a lemma which we will use extensively throughout this paper.

\begin{lem}\label{mainlem}
    Let $K$ be a Galois number field with $G = \text{Gal}(K/\bbq)$ and class group $\Cl_K$. If $G$ acts trivially on $[I] \in \Cl_K$, the order of $[I]$ divides $|G| = [K:\bbq]$.
\end{lem}

\begin{proof}
    Assume $G$ acts trivially on $[I] \in \Cl_K$. Now, by the norm property, we have 
    \[
    \prod_{\sigma \in G} \sigma(I)  \in \text{Prin}(\OK). 
    \]
    Thus, we have 
    \[  
     \prod_{\sigma \in G} [\sigma(I)] = \text{Prin}(\OK) \imp \prod_{\sigma \in G} \sigma \cdot[I] = \text{Prin}(\OK). 
    \]
    As $G$ acts trivially on $[I]$, we have 
    \[
    \text{Prin}(\OK) =  \prod_{\sigma \in G} \sigma \cdot[I] =  \prod_{\sigma \in G} [I]  = [I]^{|G|}.
    \]

    This completes the proof.
\end{proof}

The first result we present follows from Lemma 2 in \cite{cornell1981} using the representation theoretic approach. We will prove it directly here and generalize the result in Section $4$ with the techniques developed therein. 

\begin{thm}\label{cong}
    Let $K$ be a Galois number field of degree $p^r$ where $p$ is prime. Then, $h_K \equiv 0 \text{ or } 1 \text{ mod } p$.
\end{thm}

\begin{proof}
    Assume that $h_K  \not\equiv 1$ mod $p$. Once again, let $G$ act on $\Cl_K$ via the action $\sigma \cdot [I] = [\sigma(I)]$. By the orbit-stabilizer theorem, the lengths of the orbits of this group action divide $|G| = [K:\bbq] = p^r$. In particular, they must be elements of the set $\{1,p,p^2,\dots, p^r\}$. 

    Now, the identity of $\Cl_K$ has a trivial orbit of order $1$. The orbits partition $\Cl_K$, and all non-trivial orbits have order congruent to $0$ mod $p$. Hence, as we assumed $h_K  \not\equiv 1$ mod $p$, there must be some non-identity class of $\Cl_K$ in a trivial orbit. Call this element $[J] \neq \text{Prin}(\OK)$. By Lemma \ref{mainlem}, $[J]$ has order dividing $p^r$, and as we assumed $[J]$ was not the identity, $p$ must divide the order of $[J]$. Therefore $p \text{ divides } h_K$, so $h_K \equiv 0$ mod $p$. This completes the proof.
    
\end{proof}

This theorem demonstrates one of the many ways in which quadratic number fields are exceptional. Beyond their peculiarity of being Galois in general, Theorem \ref{cong} leaves all class numbers permissible in the quadratic case. The same is true of the following. 

\begin{thm}\label{greaterThanP}
    If $K$ is a Galois number field of degree $n$ and $p$ the smallest prime divisor of $n$, $h_K = 1$ or $h_K \geq p$.
\end{thm}

\begin{proof}
    Assume for the purpose of contradiction that $1 < h_K < p$. As $h_K < p$, the smallest prime divisor of $|G|$, the orbit-stabilizer theorem tells us every element of $\Cl_K$ has a trivial orbit under the action of $G$ on $\Cl_K$. As $h_K > 1$, we may take $[J] \neq \text{Prin}(\OK)$ in $\Cl_K$. Then, by Lemma \ref{mainlem}, we must have $[J]^{|G|} = \text{Prin}(\OK)$, so the order of $[J]$ divides $|G| = n$. However, this is a contradiction as $1 < |[J]| \leq h_K < p$, and $p$ is the smallest prime divisor of $n$.
\end{proof}

The question of whether every finite abelian group can be realized as the class group of a quadratic ring of integers is a famous open problem. We will consider some related questions in the next section. For now, we now turn our attention to some results on Galois number fields of odd degree. 

\begin{thm}\label{NoEven}
    Let $K$ be a Galois number field of odd degree. Then, $\Cl_K$ cannot have a unique involution.
    
\end{thm}

\begin{proof}
    Assume for the purpose of contradiction that $\Cl_K$ has a unique element of order 2, call it $[J]$. Thus, as each $\sigma \in G$ is an automorphism, we must have $\sigma \cdot [J] = [\sigma(J)] = [J]$ for all $\sigma \in G$. Hence, $G$ acts trivially on $[J]$, so by Lemma \ref{mainlem} the order of $[J]$ divides $|G|$. However, this is a contradiction as $[J]$ has order $2$ and $|G|$ is odd. 
\end{proof}

In particular, for a Galois number field of odd degree, $\Cl_K$ cannot have a unique invariant factor of even order. Notably, this precludes cyclic groups of even order. Now, from Theorem \ref{NoEven} we get the following factorization result which first appears as Lemma 3.5 in \cite{Co2}.

\begin{cor}\label{commonCor}
    Let $K$ be a Galois number field of odd degree. Then $\OK$ is an HFD if and only if $\OK$ is a UFD.
\end{cor}

\begin{proof}
    A well-known result from Carlitz (\cite{carlitz}) tells us $\OK$ is an HFD if and only if $h_K = 1$ or $2$. Noting that Theorem \ref{NoEven} disallows $h_K = 2$ and $\OK$ is a UFD if and only if $h_K = 1$ completes the proof. 
\end{proof}

Note that this result also follows directly from Theorem \ref{greaterThanP}. We conclude this section with a specific application of these results to factorization in cyclotomic number fields. These number fields in particular played a central role in the development of modern factorization theory. Being Galois in general, they provide an apt setting for an application of a generalization of the results we have developed in this section. 

\begin{thm}
    Let $p < 23$ be an odd prime and $a\in \bbz$. If $K$ is the splitting field of $x^p - a$, $\OK$ is an HFD if and only if $\OK$ is a UFD.
\end{thm}

\begin{proof}
    We begin by noting that $a$ is a perfect $p^{\text{th}}$ power then $\OK=\mathbb{Z}[\omega]$, where $\omega$ is a primitive $p^{th}$ root of unity, is known to be a UFD under the assumption that $p<23$. We will therefore assume that $a$ is not a perfect $p^{\text{th}}$ power. The splitting field of $x^p - a$, is Galois over $\bbq$ and is given by $K \cong \bbq(\omega,\alpha)$ where $\alpha$ a root of $x^p - a$. This gives us the following lattice.

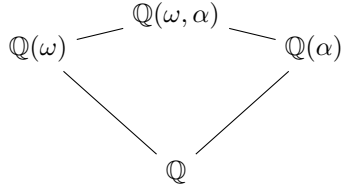
\begin{center}
\begin{tikzpicture}[node distance=1.5cm]

  \node (Q)                            {$\mathbb{Q}$};
  \node (Qw) [above left=of Q]         {$\mathbb{Q}(\omega)$};
  \node (Qa) [above right=of Q]        {$\mathbb{Q}(\alpha)$};
  \node (Qwa) [above=of Q]             {$\mathbb{Q}(\omega,\alpha)$};

  \draw (Q) -- (Qw);
  \draw (Q) -- (Qa);
  \draw (Qw) -- (Qwa);
  \draw (Qa) -- (Qwa);

\end{tikzpicture}
\end{center}

Where $\bbq(\omega,\alpha)$ is Galois over $\bbq(\omega)$ of degree $p$. Now, as $p < 23$, $\bbz[\omega]$ is a UFD, so the result follows from Corollary \ref{commonCor}.
    
\end{proof}

\section{Aut$(\Cl_K)$ and the Inverse Class Group Problem}

Recall that the action of $G = \text{Gal}(K/\bbq)$ on $\Cl_K$ induces a homomorphism from $G$ to $\text{Aut}(\Cl_K)$ given by $\sigma \rightarrow\phi$ where $\phi([I]) = \sigma \cdot [I] = [\sigma(I)]$. Note, the kernel of this homomorphism is precisely the kernel of the action. In this section, we will see how this map can be used to further restrict the structure of $\Cl_K$ for a Galois number field $K$. Conversely, we will use this tool to help answer questions related to the inverse class group problem. This terminology is used in the spirit of the inverse Galois problem. In particular, the inverse class group problem asks if an arbitrary abelian group can be realized as the class group of a ring of integers. This question was famously answered positively for the more general class of Dedekind domains by Claborn in \cite{Claborn1966}. In the setting of the imaginary quadratic rings of integers, the answer is known to be no. Certain groups can be shown by brute force not to appear, the smallest among them being $(\bbz/3\bbz)^3$. For details, the reader is encouraged to see \cite{watkins2004class}. The question remains open for rings of integers in general. In this section, for a given finite abelian group $A$, we will explore which rings of integers may admit $\Cl_K \cong A$. We begin with a direct proof of another result from \cite{cornell1981} which we will return to with some new found tools in section 4. 

\begin{thm}\label{n>1}
    Let $K$ be a cyclic extension of degree $p^m$ where $p$ is an odd prime. Then, $\Cl_K \ncong \bbz/p^n\bbz$ for $n > m$. 
\end{thm}

\begin{proof}
    Assume for the purpose of contradiction that $\Cl_K \cong \bbz /p^n \bbz$ for some $n > m$. Since $K$ is cyclic, so we have $G  \cong \bbz /p^m \bbz $, and as $p$ is odd, Aut$(\Cl_K)$ is cyclic of order $\varphi(p^n) = (p-1)p^{n-1}$. Now, the action of the Galois group on the class group induces a map from  $G$ to Aut$(\Cl_K)$. The image of this map, which will be denoted by $H$, is a cyclic subgroup of order $p^k$ for some $k\leq m$.

    If $H$ is trivial. Then, each element of $G$ acts trivially on the class group. Therefore, by Lemma \ref{mainlem}, every element of $\Cl_K$ has order dividing $|G| = p^m$, but this contradicts our assumption that $\Cl_K \cong \bbz /p^n \bbz$ with $n > m$. Hence, $|H|>1$.

     Let $\alpha$ be a primitive root mod $p^n$ and $\Cl_K = \langle [J] \rangle$. Then, the automorphism $\gamma \in \text{ Aut}(\Cl_K)$ mapping $[J] \mapsto [J]^\alpha$ generates Aut$(\Cl_K)$. Hence, $\langle \gamma^{(p-1)p^{n-1-k}} \rangle := \langle \psi \rangle $ is the unique subgroup of order $p^k$---namely $H$. Now,
         \[ 
         \prod_{\sigma \in G} \sigma(J) \in \text{Prin}(\OK) 
         \]
         implies
         \[ 
         \prod_{\sigma \in G} [\sigma(J)] =  \prod_{\sigma \in G} \sigma[(J)] = \prod_{i = 0}^{p^m-1} \psi^i([J]) = \text{Prin}(\OK) 
         \]
         Furthermore, $\psi^i([J]) = (\gamma^{(p-1)p^{n-1-k}})^i([J]) = [J]^{\left(\alpha^{(p-1)p^{n-1-k}}\right)^i}$. Note, $(p-1)p^{n-1-k} = \varphi(p^{n-k})$, so $\alpha^{(p-1)p^{n-1-k}} \equiv 1$ mod $p^{n-k}$. Also, because $\alpha$ is a primitive root mod $p^n$, each $(\alpha^{(p-1)p^{n-1-k}})^i= \alpha^{i(p-1)p^{n-1-k}}$ is unique mod $p^n$ for each $0 \leq i \leq p^m -1$. Hence, they are precisely the elements $\{1, p^{n-m} + 1, 2p^{n-m} + 1, \dots, (p^m-1)p^{n-m} + 1\}$ mod $p^n$. Taking their sum, we get 
         \[ 
         \sum_{i=0}^{p^m-1} \left(1 + i \cdot p^{n-m}\right) = p^m + p^{n-m} \sum_{i = 1}^{p^m-1} i = p^m + p^{n-m} \cdot\frac{(p^m-1)p^m}{2} \equiv p^m \text{ mod } p^n
         \]
         Hence, 
         \[  
         \prod_{i = 0}^{p^m-1} \psi^i([J]) = \prod_{i = 0}^{p^m-1} [J]^{\left(\alpha^{(p-1)p^{n-1-k}}\right)^i} = [J]^{\sum_{i=0}^{p^m-1} 1 + i \cdot p^{n-m}} = [J]^{p^m} = \text{Prin}(\OK).  
         \]
    But this implies the order of $[J]$ divides $p^m$, and this is our desired contradiction. Therefore, $\Cl_K$ cannot be $\bbz/ p^n \bbz$ for any $n > m$. 

\end{proof}

Recall that Theorem \ref{cong} tells us (in particular) that for a Galois number field of degree $p$, we must have $h_K \equiv 0$ or $1$ mod $p$. Thus, Theorem \ref{n>1} gives us a further restriction in the first case on the structure of the class group. We now shift our attention to the inverse class group. First, we consider the case where the class number $h_K$ is assumed to be an odd prime and try to place restrictions on $K$. Note, this is equivalent to asking which number fields $K$ can have $\Cl_K \cong \bbz/p\bbz$.

\begin{thm}\label{inverse}
    Let $K$ a Galois number field with $n = [K:\bbq]$ and $h_K = p$ prime. Then, $p|n$ or gcd$(p-1,n) > 1$. 
\end{thm}

\begin{proof}
    If $p=2$, the result follows directly from Theorem \ref{NoEven}, so we will assume $p$ is odd. Note that $h_K = p$ implies $\Cl_K \cong \bbz/p\bbz$. Let $\psi: G \rightarrow \text{Aut}(\Cl_K)$ be the homomorphism induced by the action of $G$ on $\Cl_K$. First, if $\text{ker}(\psi) = G$, then $G$ acts trivially on all elements of $\Cl_K$. Thus, for any non-identity element $[I] \in \Cl_K$, $|[I]| = p$ must divide $|G| = n$ by Lemma \ref{mainlem}. Alternatively, assume ker$(\psi) \lneq G$. Then, by the first isomorphism theorem, $[G: \text{ker}(\psi)]$ divides both $|G| = n$ and $|\text{Aut}(\Cl_K)| = p-1$, so gcd$(p-1,n) > 1$.
\end{proof}

This theorem tells us that having class group $\bbz/p\bbz$ is a relatively restrictive condition. For example, only Galois number fields of degree divisible by $2$ or $17$ can have class group $\bbz/17\bbz$. We continue this section with a few specific examples of the inverse class group problem which present new techniques and highlight the interesting interplay between the induced homomorphism and the norm property of the group action. 

\begin{ex}
    Let us assume $\Cl_K \cong \bbz/13\bbz$ for some Galois number field $K$. None of the theorems developed thus far preclude a cubic Galois extension from having such a class group, and indeed we see that the number field with defining polynomial $x^3-x^2-354x-2441$ is one such example (\cite[\href{https://www.lmfdb.org/NumberField/3.3.1129969.1}{Number Field 3.3.1129969.1}]{lmfdb}). Now, any cubic Galois extension has $G \cong \bbz/3\bbz$, so the induced homomorphism $\psi: G \rightarrow \text{Aut}(\Cl_K)$ must be injective as the alternative would imply there exists an element in $\bbz/13\bbz$ of order $|G| = 3$. 
    
    Now, $\text{Aut}(\bbz/13\bbz)$ is cyclic of order 12 and can be generated by $\phi$ which maps $1 \overset{\phi}{\mapsto}2$. As a permutation, we have $\phi = (1\,\,2\,\,4\,\,8 \,\,3\,\,6\,\,12\,\,11\,\,9\,\,5\,\,10\,\,7)$. Hence, the unique subgroup of order $3$ is generated by $\gamma = \phi^4 = (1\,\,3\,\,9)(2\,\,6\,\,5)(4\,\,12\,\,10)(8\,\,11\,\,7)$. Thus, $\bbz/3\bbz$ must map onto the subgroup $\{\indicator,\gamma,\gamma^2\}$ where $\indicator$ is the identity automorphism. By the norm property, we must have $a + \gamma(a) + \gamma^2(a) = 0$ for all $a \in \bbz/13\bbz$. This is equivalent to the elements in each $3$-cycle of $\gamma = (1\,\,3\,\,9)(2\,\,6\,\,5)(4\,\,12\,\,10)(8\,\,11\,\,7)$ summing to a number congruent to $0$ mod $13$ which we observe to hold. 
\end{ex}

The following theorem is inspired by this example and demonstrates that the phenomenon described at the end must occur in general for $\text{Aut}(\bbz/p\bbz)$. 

\begin{thm}
    Let $p$ be an odd prime and $\phi=(a_1\,a_2\,\ldots\,a_{p-1})$ a cyclic generator of $\text{Aut}(\bbz/p\bbz)$. If $n$ properly divides $p-1$, then the elements in each disjoint cycle of  $\phi^n$ sum to $0$ mod $p$.
\end{thm}

\begin{proof}
    Say $\phi(1) = a$, then we must have $\phi(x) = a\cdot x$ for any $x \in \bbz/p\bbz$. As $\phi$ generates $\text{Aut}(\bbz/p\bbz)$, this also implies that $a$ is a primitive root mod $p$. Now, any $x \in \bbz/p\bbz$ is contained in a $(p-1)/n$ cycle in $\phi^n$ of the form $(x\,\,\, xa^n\,\,\,x a^{2n} \dotsm x a^{(\frac{p-1}{n}-1)n})$. Taking the sum, we get
    \[
    x\sum_{r=0}^{\frac{p-1}{n}-1}(a^n)^r = x\cdot\frac{1-(a^n)^\frac{p-1}{n}}{1-a^n}= x\cdot \frac{1-a^{p-1}}{1-a^n}
    \]
    By Fermat's Little Theorem, we have $1-a^{p-1} \equiv 0$ mod $p$. The result then follows from the fact that $n$ \textit{properly} divides $p-1$, so $1-a^n \not\equiv 0$ mod $p$. 
    
\end{proof}

\begin{ex}
    Consider a Galois number field $K$ with class group $\Cl_K \cong (\bbz/2\bbz)^3$. It is well known that $\text{Aut}((\bbz/2\bbz)^3) \cong \text{GL}_3(\bbz/2\bbz)$, the group of invertible $3 \times 3$ matrices over the field of two elements with order $|\text{GL}_3(\bbz/2\bbz)| = 168 = 2^3\cdot3\cdot7$. Employing the same methods as in the proof of Theorem \ref{inverse}, we immediately see that the order of the extension $n = [K:\bbq]$ must be divisible by $2,3$ or $7$. However, Theorem \ref{cong} tells us that no Galois extension of degree $3^r$ can have class group $(\bbz/2\bbz)^3$ for any $r\geq 1$. We will now show that in fact no Galois number field of degree $3m$ where $\text{gcd}(2,m) = \text{gcd}(7,m) =1$ can have class group $(\bbz/2\bbz)^3$. 

    Assume $|G| = 3m$ where $m$ is not divisible by $2$ or $7$. Once again, the homomorphism $\psi: G \rightarrow \text{Aut}(\Cl_K) \cong \text{GL}_3(\bbz/2\bbz)$ must be non-trivial as $\Cl_K$ has no non-identity elements of order dividing $|G| = 3m$. Thus, $im(\psi)$ must be a subgroup of order $3$ by the first isomorphism theorem. 

    In the group $\text{GL}_3(\bbz/2\bbz)$, the elements of order $3$ lie in a unique conjugacy class which can be represented by the matrix $E$ below.

    \[
    E = \begin{bmatrix}
    1 & 0 & 0 \\
    0 & 0 & 1 \\
    0 & 1 & 1
    \end{bmatrix}
    \]

In particular, $E$ (and hence any of its conjugates) has a nontrivial eigenvector corresponding to the eigenvalue $1$ in $(\bbz/2\bbz)^3$. If we denote by $[I]$ the ideal class corresponding to this eigenvector for the matrix $E'$ that generates $im(\psi)$, we can see that $[I]$ is fixed by all elements of $\langle E'\rangle$ and so, by Lemma \ref{mainlem}, we have that $|[I]|=2$ divides $|G|$ which is our desired contradiction.
\end{ex}

Similar methods can be employed to investigate the inverse class group problem for $p$-elementary abelian groups as well as cyclic groups of order $2^mp^r$ with $m\in \{0,1\}$ and $r \geq 0$ where $p$ is an odd prime. Note, along with $n =4$, the latter are precisely the values of $n$ for which $\text{Aut}(\bbz/n\bbz)$ is cyclic.

\section{Localization and the Class Group} 

In this section, we will use localization to strengthen some of the results developed previously. First, recall that any overring---in particular, any localization---of a Dedekind domain remains Dedekind. Intuitively, as a localization of $R$ is attained by turning a set of elements $S\subseteq R$ into units, we expect the class group of $R_S$ to be smaller than that of $R$ as those non-principal ideals $I \subseteq R$ for which $I \cap S \neq \emptyset$ will become principal. As we will demonstrate, \enquote{small} localizations will tend to lead to small changes in the class group. In the context of rings of algebraic integers, this will be useful in further restricting the possible structure of $\Cl_K$. The following result can be viewed as a particular case of Theorem 1.5 in \cite{Claborn1965}---one of Claborn's pioneering works on Dedekind domains. 

It is useful at the outset to recall that $D[\frac{1}{\alpha}]$ is the localization of $D$ at the multiplicative set $\{\pm\alpha^n\,|\,n \in\bbn\}$ which is the complement of the set-theoretic union of all prime ideals not containing $\alpha$. 

\begin{thm}\label{local}
    Let $D$ be a Dedekind domain and $\alpha \in D$ a nonzero nonunit with 
    \[
    (\alpha) = \p_1^{n_1}\p_2^{n_2}\dotsm\p_k^{n_k}.
    \]
    Then $\Cl(D[\frac{1}{\alpha}]) \cong \Cl(D)/\left<[\p_1],[\p_2],...,[\p_k]\right>$. 
\end{thm}


\begin{proof}

    Let $\psi: \Cl(D) \rightarrow \Cl(D[\frac{1}{\alpha}])$ be the canonical map $[I] \overset{\psi}{\mapsto} \left[I D[\frac{1}{\alpha}]\right]$. An application of Propositions 1.3 and 1.5 of \cite{Claborn1965} gives that $\psi$ maps $\Cl(D)$ onto $\Cl(D[\frac{1}{\alpha}])$ and the kernel is generated by the classes $[\p_\beta]$ where the $\p_\beta$ range over the prime ideals such that $\p_\beta D[\frac{1}{\alpha}] = D[\frac{1}{\alpha}]$. As $D[\frac{1}{\alpha}] = D_S$ where $S = \{\pm\alpha^n\, |\, n \in \bbn\}$, these are precisely the prime ideals $\p$ such that $\p\cap S \neq \emptyset$ or the primes containing $\alpha^n$ for some $n \in \bbn$. By observing that these are the prime divisors of $(\alpha)$, we obtain the result by the first isomorphism theorem. 
\end{proof}

Notably, if $\Cl(D)$ is finite, Theorem \ref{local} implies that any homomorphic image of $\Cl(D)$ can be realized as the class group of such a localization of $D$. Claborn gives a similar, slightly weaker result in \cite{Claborn1965} where, given a Dedekind domain $A$ with arbitrary class group $G$ and subgroup $H$, he constructs a related Dedekind domain with class group isomorphic to $G/H$. As he answered the inverse class group problem for Dedekind domains, Theorem \ref{local} brings us a step closer for overrings of $\OK$. In essence, $\OK[\frac{1}{\alpha}]$ is very nearly a ring of integers itself. In fact, it is the precise analog if the base ring $\bbz$ were replaced by $\bbz[\frac{1}{N(\alpha)}]$ for some nonzero $\alpha \in \mathcal{O}_K$.

\begin{thm}\label{replace}
    Let $K$ be a number field and $\alpha \in \mathcal{O}_K$ a nonzero element. Then $\mathcal{O}_K[\frac{1}{\alpha}]$ is the integral closure of $\bbz[\frac{1}{N(\alpha)}]$ in $K$. 
\end{thm}

\begin{proof}
    As an overring of a Dedekind domain, $\mathcal{O}_K[\frac{1}{\alpha}]$ is also Dedekind and thus integrally closed in $K$. Thus, we need only show that $\mathcal{O}_K[\frac{1}{\alpha}]$ is integral over $\bbz[\frac{1}{N(\alpha)}]$; it is sufficient to show $\frac{1}{\alpha}$ is integral over $\bbz[\frac{1}{N(\alpha)}]$. Let $f(x) \in \bbz[x]$ be the minimal polynomial of $\alpha$; we write $f(x) = x^n + m_{n-1}x ^{n-1}+ \dotsm + m_1x \pm N(\alpha)$. Thus,
    \[
    \pm N(\alpha) + m_1\alpha + \dotsm + m_{n-1}\alpha^{n-1} + \alpha^n = 0.
    \]
    Multiplying by $\pm\frac{1}{N(\alpha)\alpha^n}$ yields
    \[
    \left(\frac{1}{\alpha}\right)^n \pm \left( \frac{m_1}{N(\alpha)}\left(\frac{1}{\alpha}\right)^{n-1} + \dotsm + \frac{m_{n-1}}{N(\alpha)}\left(\frac{1}{\alpha}\right)+ \frac{1}{N(\alpha)}\right) = 0.
    \]
    This completes the proof. 
\end{proof}




Note also that $\Cl(\OK[\frac{1}{\alpha}])$ is finite and retains the useful property that each ideal class contains a prime ideal (this follows directly from the correspondence between prime ideals of $\OK[\frac{1}{\alpha}]$ and prime ideals of $\OK$ disjoint from $\{\pm\alpha^n,\,n \in\bbn\}$). This allows for the application of many theorems on rings of integers to be applied to $\OK[\frac{1}{\alpha}]$. The following gives one such example.

\begin{ex}
    Consider the ring of integers $\OK = \bbz[\sqrt{-14}]$ with class group $\Cl_K \cong \bbz/4\bbz$ generated by $\p = (3,1-\sqrt{-14})$. Writing $\q = (2,\sqrt{-14})$, it is not difficult to show $[\q] = [\p^2]$. Now $\q^2=(2)$, so by Theorem \ref{local} we have $\Cl(\OK[\frac{1}{2}]) \cong\bbz/2\bbz$ implying $\OK[\frac{1}{2}] = \bbz[\sqrt{-14}, \frac{1}{2}]$ is an HFD which is not a UFD. 
\end{ex}

Most significantly, Theorem \ref{local} allows us to excise a minimal amount of $\Cl_K$ and thus produce relatively large homomorphic images. Once again, we can realize any homomorphic image of $\Cl_K$ in this way. In general, if $\gamma: \Cl_K \twoheadrightarrow H$ is a surjective homomorphism with $\text{ker}(\gamma) = \left<[\p_1],[\p_2],...,[\p_n]\right>$, $\Cl(\OK[\frac{1}{\alpha}]) \cong H$ where $(\alpha) = (\p_1\p_2\dotsm\p_n)^{h_K}$. 

The remainder of this section will be dedicated to leveraging this fact to extend the techniques and theorems developed in sections 2 and 3. One would hope that the Galois group of $K$ acts on the class group of $\OK[\frac{1}{\alpha}]$ in the same manner as $\Cl_K$ which would allow us to place the same restrictions on the homomorphic images of the class group---further constraining the structure of $\Cl_K$. Unfortunately, the action $\sigma\cdot[\mathfrak{P}]=[\sigma(\mathfrak{P})]$ for $\sigma \in G$ and $[\mathfrak{P}]\in \Cl(\OK[\frac{1}{\alpha}])$ will not be well-defined in general as $\sigma(\frac{1}{\alpha})=\sigma(\alpha)^{-1}$ need not be in $\OK[\frac{1}{\alpha}]$ in general (Example \ref{gex} in the next section also illustrates this). However, in the case that $\alpha \in \bbz$, we avoid this issue as we will have $\sigma(\frac{1}{\alpha})=\frac{1}{\alpha}$ for all $\sigma \in G$, and this simple observation will prove useful presently.

\begin{cor}\label{local2}
    Let $K$ be a Galois number field and suppose that $\alpha \in \OK$ is a nonzero nonunit with prime ideal factorization
    \[
    (\alpha) = \p_1^{n_1}\p_2^{n_2}\dotsm\p_k^{n_k}.
    \]
    Then $\Cl(\OK[\frac{1}{N(\alpha)}]) \cong \Cl_K/G\left<[\p_1],[\p_2],...,[\p_k]\right>$ where $N(\alpha)$ is the norm of $\alpha$, and $G\left<[\p_1],[\p_2],...,[\p_k]\right>$ denotes the subgroup generated by the images of the ideal classes $[\p_i]$ under the Galois action.
\end{cor}

\begin{proof}
    Recall that the ideal norm agrees with the element norm up to a unit for principal ideals. Thus, as $K$ is Galois, we have $(N(\alpha)) = N((\alpha)) = \prod_{\sigma \in G}\sigma((\alpha)) = \prod_{\sigma \in G}\sigma(\p_1)^{n_1}\dotsm \sigma(\p_k)^{n_k}$. The theorem then follows by applying Theorem \ref{local} to $N(\alpha)$.
\end{proof}

Now, the group action of $G$ on the class group of $\OK[\frac{1}{N(\alpha)}]$ will be well-defined. More than this, it will be a norm-like Galois action. 

\begin{thm}\label{norm-like}
    For a Galois number field $K$ and nonzero rational integer $n$, the group action $\sigma \cdot [\mathfrak{P}] = [\sigma(\mathfrak{P})]$ for $\sigma \in \text{Gal}(F/\bbq), [\mathfrak{P}] \in \Cl\left(\OK[\frac{1}{n}]\right)$ is a well-defined, norm-like action. 
\end{thm}

\begin{proof}
    It is sufficiently clear that conditions 1-3 from Definition \ref{maindef} hold. Now, let $[\mathfrak{P}] \in \Cl(\OK[\frac{1}{n}])$ where $\mathfrak{P}$ is a prime ideal representative. Note, such a $\mathfrak{P}$ must exist because each class of $\Cl_K$ contains infinitely many primes and only finitely many intersect $\{\alpha^n \,|\, n\in\bbn$\}. Thus, existence follows from the one to one correspondence between prime ideals of $\mathcal{O}_K$ disjoint from $S=\{\alpha^n \,|\, n\in\bbn\}$ and prime ideals of $\mathcal{O}_K[\frac{1}{\alpha}]$ given by $\p \longleftrightarrow \p \mathcal{O}_K[\frac{1}{\alpha}]$ for $\p \in \text{Spec}(\mathcal{O}_K)$. 
    
    If we write $\p = \mathfrak{P} \cap \OK$, we claim that $\left(\prod_{\sigma\in G}\sigma(\p)\right)\OK[\frac{1}{n}] = \prod_{\sigma\in G}\sigma(\mathfrak{P})$, and thus the norm property (4) is also inherited from our original group action. The veracity of this is immediate from the facts that if $S$ is a multiplicatively closed subset of $R$ and $I,J\subset R$ are ideals then $(IJ)_S=I_SJ_S$ and if $P\subset R$ is prime then $P=P_S\bigcap R$.
 \end{proof}


We now note that any restrictions the structure of $G$ places on $\Cl_K$ must also hold for $\Cl(\mathcal{O}_K[\frac{1}{n}])$, in particular, on $\Cl(\mathcal{O}_K[\frac{1}{N(\alpha)}])$ for any nonzero, nonunit $\alpha$. Thus, by Corollary \ref{local2}, if the set $\{[\p_1],...,[\p_k]\}$ is closed under the action of $G$, then we may apply our previous methods to $\Cl_K/\left<[\p_1],...,[\p_k]\right>$. This is a good example of how misleading the quadratic case can be, as in this case $\text{Gal}(K/\bbq)$ induces the subgroup $\{\indicator, -\indicator\} \leq \text{Aut}(\Cl_K)$ where $\indicator$ is the identity and $-\indicator([I]) = [I]^{-1}$, we see that any homomorphic image of $\Cl_K$ can be realized by localizing at an integer---allowing significant restriction on $\Cl_K$. 

In general, it will be sufficient for a subgroup to be $\text{Aut}(\Cl_K)$-invariant. For example, if $G$ cannot admit a Galois action on $\bbz/n\bbz$, then we cannot have $\Cl_K \cong \bbz/an\bbz$ for any $a \in \bbn$. This follows from the fact that $\bbz/an\bbz$ has a unique subgroup of order $a$, call it $H$, and the quotient $(\bbz/an\bbz)/H \cong \bbz/n\bbz$. With this in hand, Theorem \ref{NoEven} becomes an immediate corollary of Theorem 3.1 in \cite{Co2}. Furthermore, we can quickly improve upon Theorem \ref{n>1}. 


\begin{cor}
    Let $K$ be cyclic of degree $p^m$ where $p$ is an odd prime. Then, $\Cl_K \not\cong \bbz/np^{m+1}\bbz$ for any $n \in \bbn$.
\end{cor}

Once again, as we must have $h_K\equiv0$ or $1$ mod $p$, this constitutes a serious restriction for extensions of (odd) prime degree. Now, note also that any Sylow $q$-subgroup of $\Cl_K$ is invariant under automorphism. Recalling that a group is the direct product of its Sylow $q$-subgroups, by the same logic, any restriction placed on $\Cl_K$ may also be applied to its Sylow $q$-subgroups. Applying this also to Theorem \ref{n>1}, the following stronger result which appears in \cite{Lemmermeyer2003} is immediate.

\begin{cor}
    Let $K$ be cyclic of degree $p^m$ where $p$ is an odd prime. Then, the Sylow $p$-subgroup of $\Cl_K$ is not isomorphic to $\bbz/p^n\bbz$ for any $n > m$.
\end{cor}

In the same way, we can greatly strengthen Theorem \ref{cong}.

\begin{cor}\label{cong2}
    Let $K$ be a Galois number field of degree $p^r$ and $S(q)$ a =non-trivial Sylow $q$-subgroup of $\Cl_K$. Then, $p=q$ or $|S(q)| \equiv 1$ mod $p$. Therefore, if $h_K = p^r q_1^{n_1}\dotsm q_k^{n_k}$, then $q_i^{n_i} \equiv 1$ mod $p$ for all $1\leq i \leq k$.
\end{cor}

This was proven for abelian number fields by Fr{\"o}hlich (\cite{Frolich1952}) which is a notably more restrictive condition. Now, observe that while Theorem \ref{cong} would not preclude an extension of degree $[K:\bbq]=3^r$ from having $h_K = 55$, Corollary \ref{cong2} shows this is impossible. In fact, for such an extension to have $h_K = 5n$ it is necessary for the $5-$adic valuation of $5n$ to be even.

\section{Overrings of Dedekind Domains}

We proceed with a general consideration of overrings of Dedekind domains with eye toward developing a general criterion for when the class group of an overring of $\mathcal{O}_K$ is a Galois module via the action we have been studying. In particular, for an overring $\mathcal{O}_K \subseteq R \subseteq K$, we would like to know when the Galois group preserves fractional ideals of $R$ and thus admits the same norm-like group action on $\Cl(R)$. We saw in Section 5 that even in the case of simple extensions, this may fail. Consider the following example. 


\begin{ex}\label{gex}
    Let $\mathcal{O}_K = \bbz[\sqrt{-5}]$ and consider the overring $R= \mathcal{O}_K\left[\frac{1}{1+\sqrt{-5}}\right]$. Now, $\text{Gal}(K/\bbq) = \{\indicator, \sigma\}$ where $\sigma$ represents complex conjugation. Note that $\sigma(R)$ fails to be an $R$-module as $\frac{1}{1+\sqrt{-5}}\cdot \sigma\left( \frac{1}{1+\sqrt{-5}}\right) = \frac{1}{1+\sqrt{-5}} \cdot \frac{1}{1-\sqrt{-5}} = \frac{1}{6} \notin R$.
\end{ex}

The following is a well-known characterization of overrings of Dedekind domains which first appeared as Theorem 1 in \cite{GILMER1966331}. 

\begin{prop}\label{DedekindOver}
    Let $D$ be a Dedekind domain and $R$ an overring of $D$. Then
    \[
    R = \underset{\p R \neq R}{\bigcap}D_{\p}
    \]
    where the $\p$ range over $\text{Spec}(D)$.
\end{prop}

The following result can be found in \cite{grams1974distribution}, and it follows quickly from the preceding proposition. We include a proof here as it is constructive and will inform later results and examples.

\begin{thm}\label{correspondence}
    Let $D$ be a Dedekind domain with torsion class group. There exists a one to one correspondence between overrings of $D$ and nonempty subsets $S\subseteq\text{Spec}(D)$ given by 
    \[
    S \longleftrightarrow \underset{\p \in S}{\bigcap}D_{\p}.
    \]
\end{thm}

\begin{proof}
    By Proposition \ref{DedekindOver}, every overring of $D$ can be expressed in the form $\underset{\p \in S}{\bigcap}D_{\p}$ for some $S \subseteq \text{Spec}(D)$. Thus, it suffices to show that $\underset{\p \in S_1}{\bigcap}D_{\p} \neq\underset{\p \in S_2}{\bigcap}D_{\p}$ for subsets $S_1,S_2 \subseteq \text{Spec}(D)$ whenever $S_1 \neq S_2$. Without loss of generality, assume $\q \in S_1 \backslash S_2.$ Recall that $\Cl(D)$ is torsion, so $|[\q]| := n$ is finite. Thus, $\q^n = (q)$ for some $q \in D$. Now, $\q$ is the unique prime ideal containing $q$, so $\frac{1}{q} \in \underset{\p \in S_2}{\bigcap}D_{\p}\backslash \underset{\p \in S_1}{\bigcap}D_{\p}$.
\end{proof}

\begin{cor}
    Let $D$ be a Dedekind domain with torsion class group. Then $D \subseteq D[\frac{1}{a}]$ is a adjacent ring extension if and only if $aD = \p^k$ for some $k \in \bbn$. 
\end{cor}

\begin{proof}
    We prove the forward implication by contraposition. Assume there are distinct primes $\p,\q \in \text{Spec}(D)$ which divide $aD$. Then, $a \in \p,\q$, and hence $\p D[\frac{1}{a}] = \q D[\frac{1}{a}] = D[\frac{1}{a}]$. Thus, by Theorem \ref{correspondence}, there is an overring properly between $D$ and $D[\frac{1}{a}]$.

    Now, for the reverse implication, assume $aD = \p^k$, $k$ is minimal respect to the property that $\p^k$ is principal, and $\q D[\frac{1}{a}] = D[\frac{1}{a}]$ for some prime $\q \in \spec(D)$. Because $aD = \p^k$, $a$ is strongly irreducible in $D$, and thus the saturation of $\{a^n \,|\, n \in \bbn\}$ in $D$ is simply $\{a^n \,|\, n \in \bbn\}$. Hence, $a^k \in \q$ for some $k \in \bbn$ which implies $a \in \q$. However, $\p$ is the unique prime ideal containing $a$, so $\p = \q$. Therefore, $\p$ is the unique prime ideal of $D$ such that $\p D[\frac{1}{a}] = D[\frac{1}{a}]$. By the correspondence in Corollary \ref{correspondence}, we have the $D \subseteq D[\frac{1}{a}]$ must be adjacent. 

\end{proof}

By the correspondence in Theorem \ref{correspondence}, we see that for any adjacent ring $R $ over $D$, there is a unique $\p \in \text{Spec}(D)$ such that $\p R = R$. What is more, recalling that any adjacent extension is simple, $\p$ uniquely determines the irreducible $a \in D$ (up to a unit) such that $R = D[\frac{1}{a}]$. In particular, we have $(a) = \p^n$ where $n = |[\p]|$. Inspired by this uniqueness, we present the following Jordan--H\"older analog for sequences of adjacent overrings.

\begin{thm}\label{jordan}
    Let $D$ be a Dedekind domain with finite class group $\Cl(D)$ such that every class contains a prime ideal. Let $R$ be an overring of $D$ which is minimal with respect to being a PID. If $D := D_0 \subseteq D_1 \subseteq D_2 \subseteq \dotsm \subseteq  D_{r_1} := R$ and $D := D_0' \subseteq D_1' \subseteq D_2' \subseteq \dotsm \subseteq  D_{r_2}' := R$ are sequences of adjacent domains, then $r_1 = r_2$ and the sequence of $\alpha_i \in R$ such that $D_{i+1} = D_i[\alpha_i]$  is unique up to reordering. What is more, such an overring is guaranteed to exist. 
\end{thm}
    
\begin{proof}

    We start by noting that for any adjacent domain $D \subseteq D_1$ there exists a unique $\p \in \text{Spec}(D)$ such that $\p D_1 = D_1$ by Theorem \ref{correspondence}. Thus, we must have $D_1 = \cap_{\q \neq \p} D_\q = D[\frac{1}{a}]$ where $aD = \q^n$ and $n = |[\q]|$. Recall that $D[\frac{1}{a}]$ is the localization of $D$ at the complement of the set-theoretic union of all prime ideals of $D$ not containing $a$ (in this case, that is all prime ideals except $\q$). By continuing this process and recalling the correspondence between prime ideals of a localization $D_S$ and prime ideals of $D$
    disjoint from $S$ given by $\p \mapsto \p D_S$, we see that there is a unique, finite set of $r_1$ primes $\p_1,...,\p_{r_1} \in \text{Spec}(D)$ such that $\p_i R = R$. Thus, any sequence of adjacent domains between $D$ and $R$ must have length $r_1 :=r$ and is unique up to our ordering of the $\p_i$. In particular, if $\p_{i}D_{i-1}$ is a prime ideal such that $\p_{i}D_i = D_i$, we must have $D_i = D_{i-1}[\frac{1}{a_i}]$ where $a_iD_{i-1} = \p_i^{n_i}$.

    Now, to construct such a PID and sequence, take some non-principal $\q_1 \in D$ such that $\q_1^{n_1} = b_1 D$, and let $D_1 = D[\frac{1}{b_1}]$. As $[\q_1] \neq \text{Prin}(D)$, we must have $|\Cl(D_1)| < |\Cl(D)|$. By the finiteness of the class group, if we continue this process, we will eventually have $|\Cl(D_t)|= 1$ for some $t  \in \bbn$ with $|\Cl(D_i)| \neq 1$ for all $0 < i < t$. Thus, $D_t$ will be minimal with respect to being a PID.
\end{proof}


We will now shift our focus back to the particular case of rings of integers. For a more general treatment of overrings of Dedekind domains and their class groups, the reader is encouraged to see \cite{Claborn1966} and \cite{grams1974distribution}. Theorem \ref{jordan} tells us that any ring of integers in particular is finitely many adjacent steps away from being a PID. In this setting, this is equivalent to asking how far the ring is from being a UFD---a question considered more broadly in \cite{anderson2010far}.

\begin{ex}
    Recall $\Cl_K \cong \bbz/4\bbz$ for $\OK = \bbz[\sqrt{-14}]$, and $\Cl_K$ is generated by $[\p_3]$ where $\p_3$ is a prime lying over $3$. Now, if $\p_2 = (2,\sqrt{-14})$, then $\p_2^2 = (2)$ and we must have $\Cl(\bbz[\sqrt{-14},\frac{1}{2}]) \cong \bbz/2\bbz$. Note, this will be a non-UFD HFD. Furthermore, as $|[\p]| =4$, we must have that $\mathfrak{P}:= \p\bbz[\sqrt{-14},\frac{1}{2}]$ is a non-principal prime ideal of the overring. Hence, if $\alpha\bbz[\sqrt{-14},\frac{1}{2}] = \mathfrak{P}^2$, $\bbz[\sqrt{-14},\frac{1}{2},\frac{1}{\alpha}]$ is a minimal PID over $\OK$.
\end{ex}

Now, even in the case of rings of integers, $R$ being a minimal PID over $\mathcal{O}_K$ does not imply the existence of such chain of adjacent domains. In fact, the following example implies that there always exists a PID over $\mathcal{O}_K$ failing to have this property except in the case that the ring of integers itself is a PID.

\begin{ex}
    Let $R$ be a ring of integers with $\Cl_K = \bbz/2\bbz$, and let $P$ be the set of all non-principal prime ideals of $R$. Now, taking $S=(\underset{\p \in P}{\cup}\p)^c$, we have that $R_S = \{\frac{r}{s} \, | \, r \in R, s\in S\}$ also has class group $\bbz/2\bbz$ by Theorem 1.5 of \cite{claborn1965dedekind}. Note, this is equivalent to localizing at the multiplicatively closed set generated by the prime elements of $R$. Now, take a (non-principal) prime $\q \in \text{Spec}(R_S)$, and let $aR_S = \q^2$ for some $n \ \in \bbn$. Then, $R_S[\frac{1}{a}]$ is minimal over $R$ with respect to being a PID (in the sense that there is no PID strictly between $R$ and $R_S[\frac{1}{a}]$ that is comparable to $R_S$), but there is no finite chain of adjacent domains between $R$ and $R_S[\frac{1}{a}]$.
\end{ex}

We now return to the particular case of rings of integers. For a more general treatment of overrings of Dedekind domains, the reader is referred to \cite{claborn1965dedekind} and \cite{Claborn1965}. In the following theorem, we adopt the notation of \cite{Nark2004} by denoting the group of fractional ideals of a Dedekind domain $R$ as $G(R)$. 

\begin{thm}\label{invariant}
    Let $K$ be a Galois number field with $G := \text{Gal}(K/\bbq)$ and $R$ an overring of the ring of integers $\mathcal{O}_K$. The following are equivalent:

    \begin{enumerate}
        \item The action $\sigma \cdot [I] = [\sigma (I)]$ for $\sigma \in G$ and $[I] \in G(R)$ is well-defined. 

        \item $R$ is Galois-invariant. That is, for all $\sigma \in G$, $\sigma(R) = R$.

        \item For all $\p \in \text{Spec}(\mathcal{O}_K)$ and $\sigma \in G$, $\p R = R \text{ if and only if } \sigma(\p)R = R$. 

        \item For any $\p, \q \in \text{Spec}(\mathcal{O}_K)$ lying over the same rational prime $p$, $\p  R = R \text{ if and only if } \q R = R$.
    \end{enumerate}
\end{thm}

\begin{proof}
    $(1 \Rightarrow 2)$ By well-definedness, we must have $\sigma(R)$ an $R$-module for all $\sigma \in G$. Then, for all $r \in R$, $r\cdot \sigma(1) = r \cdot 1 = r \in \sigma(R)$, so we have $R \subseteq \sigma(R)$. By the same argument, $R \subseteq \sigma^{-1}(R)$ and thus $\sigma(R) \subseteq R$. Therefore, $\sigma(R) = R$ as desired. 

    $(2 \Rightarrow 1)$ Assume $R$ is Galois-invariant. Let $I \in G(R)$ and $\sigma \in G$. By definition, we have $\alpha I \subseteq R$ for some $\alpha\in K$. By our assumption, $\sigma(\alpha)\sigma(I) = \sigma(\alpha I) \subseteq \sigma(R) \subseteq R$ for $\sigma(\alpha) \in K$. Now, let $r \in R$, then $r' :=\sigma^{-1}(r) \in R$, and note that $r \sigma(I) = \sigma(r')\sigma(I) \in \sigma(r' I) \subseteq \sigma(I)$ as $I$ is an $R$-module. We conclude that $\sigma(I)$ is a fractional ideal, and thus the map $I \mapsto \sigma(I)$ on $G(R)$ descends to the group action in $1$.

    $(2 \Rightarrow 3)$ We will proceed by contraposition. Assume that there exists $\p \in \text{Spec}(\OK)$ and $\sigma \in G$ such that $\p R = R$ but $\sigma(\p)R \neq R$. Then there exists $\frac{a}{b} \in U(R) \cap \p R$, so $\frac{b}{a} \in R$. However, $\sigma(\frac{b}{a}) = \frac{\sigma(b)}{\sigma(a)} \notin R$ because $\sigma(\frac{a}{b}) = \frac{\sigma(a)}{\sigma(b)} \in \sigma(\p)$ and $\sigma(\p) \cap U(R) = \emptyset$. Therefore, $R$ is not Galois-invariant.

    $(3 \Rightarrow 2)$ We prove this directly. Assume (3) holds and consider $\sigma(R)=\sigma(\cap_{\p R \neq R} (\OK)_\p)=\cap_{\p R \neq R} (\sigma(\mathcal{O}_K))_{\sigma(\p)}=\cap_{\p R \neq R} (\mathcal{O}_K)_{\sigma(\p)}$. Since $\p R \neq R$ if and only if $\sigma(\p ) R \neq R$, the proof is complete.

    $( 3 \Leftrightarrow4)$ The equivalence of these conditions follows directly from the transitive action of $G$ on the primes lying over a given rational prime (for details, see \cite{Marcus2018}). This completes the proof.
\end{proof}

This theorem demonstrates the fact that any class group of an overring of $\mathcal{O}_K$ which admits such a group action can be realized by a simple extension of the form $\mathcal{O}_K[\frac{1}{n}]$ where $n$ is an integer. In particular, we see that the types of extensions studied in the previous section are sufficient to achieve all homomorphic images which are of use to us in restricting the structure of $\Cl_K$. The theorem also provides a general framework for using this approach to understanding particular class groups. Consider the following corollary.

\begin{cor}\label{cosets}
    Let $K$ be a Galois number field of degree $p^r$ and $n$ an integer. If $H$ is the subgroup of $\Cl_K$ generated by the classes of primes lying over $n$, then $[\Cl_K:H] \equiv 0$ or $1$ mod $p$.
\end{cor}

\begin{proof}
    Theorem \ref{norm-like} tells us that $\text{Gal}(K/\bbq)$ admits a norm-like action on the class group of $\mathcal{O}_K[\frac{1}{n}]$ which will have order $[\Cl_K:H]$. The result then follows directly from a modest generalization of Theorem \ref{cong}. 
\end{proof}


\section{The Arithmetic of the Normset}

In this section, we consider some arithmetic implications of the action we have been considering. We will begin by recalling and further developing the notion of the Davenport constant of a weighted zero-sum sequence monoid from a number theoretic point of view, leading to a nice arithmetic interpretation as well as some illustrative and novel examples. We conclude by making a connection between the arithmetic of quadratic norms and the partition problem---a classic NP-complete problem. 

If $K$ is a Galois number field, the action we have been studying plays a central role in understanding the arithmetic of the normset $N_K := \{|N(\alpha)| \, : \, \alpha \in \mathcal{O}_K^*\}$. This connection was first made in \cite{C96b} where the author proved that $N_K$ is saturated if and only if $\text{{Gal}}(K/\bbq)$ acts trivially on $\Cl_K$. This connection was greatly strengthened with the advent of weighted Davenport constants. Let $G$ be a finite abelian group. Recall that the Davenport constant of $G$, denoted $D(G)$, is the minimum natural number $n$ such that any $G$-sequence $\{g_1,g_2,...,g_n\}$ of (not necessarily distinct) elements of $G$ must have a zero-sum subsequence---that is, a subsequence summing to the identity. Equivalently, we may define it as the largest natural number $m$ such that there exists a zero-sum sequence $\{g_1,g_2,...,g_m\}$ with no proper zero-sum subsequence.

Now, for a finite abelian group $G$ and $\Gamma$ a set which operates on $G$ by endomorphism, $D_\Gamma(G)$ is the smallest positive integer $n$ such that given any $G$-sequence $\{g_1,g_2,...,g_n\}$, there exists a subsequence $\{g_{i_1},g_{i_2},...,g_{i_k}\}$ and weights $\sigma_{i_1},\sigma_{i_2},...,\sigma_{i_k} \in \Gamma$ such that $\{\sigma_{i_1}(g_{i_2}),\sigma_{i_2}(g_{i_2}),...,\sigma_{i_k}(g_{i_k})\}$ is a zero-sum sequence. Note that \enquote{operating by endomorphism} means precisely meeting condition 3 of Definition \ref{maindef} and thus inducing a map into $\text{End}(G)$. We refer to $D_\Gamma$ as a weighted Davenport constant. This definition closely mimics the the definition of $D(G)$, and in fact $D_{\{\indicator\}}(G) = D(G)$. What is more, just as the Davenport constant is intimately related to the arithmetic of $\mathcal{O}_K$, the weighted Davenport constant is closely related to that of $N_K$.

Weighted Davenport constants were first studied by Adhikari et al in \cite{adhikari2006davenport}, but it was Halter-Koch who first made the connection with the ideal arithmetic of rings of integers in \cite{halter2014arithmetical}. In particular, if we allow $\Gamma = \text{Gal}(K/\bbq)$ for a Galois number field $K$, Theorem 2.2 of \cite{halter2014arithmetical} gives a nice characterization of $D_\Gamma(\Cl_K)$ in terms of norms of ideals in $\mathcal{O}_K$.

There has been a recent resurgence of interest in this area through the study of monoids of weighted zero-sum sequences (\cite{BMOS22}, \cite{GHKZ22}). Given a subset $\Gamma \subseteq \text{End}(G)$, a $G$-sequence $\{g_1,g_2,...,g_n\}$ is called a $\Gamma$-weighted zero-sum sequence if there exist $\sigma_1,\sigma_2,...,\sigma_n \in \Gamma$ such that $\{\sigma_1(g_1),\sigma_2(g_2),...,\sigma_n(g_n)\}$ is a zero-sum sequence. The set of all weighted zero-sum sequences forms a monoid under concatenation which is denoted $\mathfrak{B}_\Gamma(G)$. With this notation, we see that $D_\Gamma(G)$ is the minimal $n$ such that any $G$-sequence must have a $\Gamma$-weighted zero-sum subsequence. 

A $\Gamma$-weighted zero-sum sequence $\{g_1,g_2,...,g_n\}$ is called minimal if it cannot be partitioned into two nonempty $\Gamma$-weighted zero-sum sequences. Then $D(\mathfrak{B}_\Gamma(G))$ denotes the largest positive integer $m$ such that there exists a minimal $\Gamma$-weighted zero-sum sequence of length $m$. Note that a $\Gamma$-weighted zero-sum sequence $\{g_1,g_2,...,g_n\}$ may be minimal and have a proper $\Gamma$-weighted zero-sum subsequence. Hence, we see that $D_\Gamma(G)$ and $D(\mathfrak{B}_\Gamma(G))$ are related but distinct. 

For the remainder of this paper, we will focus on the case where $\Gamma := \text{Gal}(K/\bbq)$ for a Galois number field $K$. In this case, the authors of \cite{BMOS22} showed that there exists a transfer homomorphism from $N_K$ to $\mathfrak{B}_\Gamma(\Cl_K)$. In essence, this means that factorization properties of the normset are completely characterized by the arithmetic of  these $\Gamma$-weighted zero-sum sequences---that is, by the action of $\text{Gal}(K/\bbq)$ on $\Cl_K$. Extensive results on $\mon$ and implications for the factorization properties of $N_K$ are studied through this lens in \cite{BMOS22} and \cite{GHKZ22}. We will presently develop the connection between these two objects from a strictly number-theoretic point of view. This will allow us to give a nice arithmetic interpretation of $D(\mathfrak{B}_\Gamma(\Cl_K))$ and provide some enlightening new examples.

First, let us recall that for a Galois number field $K$ and $\alpha \in \mathcal{O}_K$, $N(\alpha) = \prod_{\sigma\in \text{Gal}(K/\bbq)}\sigma(\alpha)$. Thus, if we write $(\alpha) = \p_1\p_2\dotsm \p_m$ where the prime ideals are not necessarily distinct, it is not hard to see that $N(\beta) = N(\alpha)$ if and only if there exist $\bar{\sigma}_1,\bar{\sigma}_2,...,\bar{\sigma}_m \in \text{Gal}(K/\bbq)$ such that $(\beta) = \bar{\sigma}_1(\p_1)\bar{\sigma}_2(\p_2)\dotsm \bar{\sigma}_m(\p_m)$. This follows from the uniqueness of prime ideal factorization as $N(\alpha)$ agrees with the ideal norm, and $(N(\alpha)) = \prod_{\sigma \in \text{Gal}(K/\bbq)}\sigma((\alpha)) = \prod_{\sigma \in \text{Gal}(K/\bbq)}\sigma(\p_1)\sigma(\p_2)\dotsm \sigma(\p_m)$. Note that $\{[\p_1],[\p_2],...,[\p_m]\}$ is a zero-sum sequence in $\Cl_K$, so finding such an element $\beta$ corresponds to finding weights $\phi_1,\phi_2,...,\phi_m$ in the subgroup of $ \text{Aut}(\Cl_K)$ induced by the action of $\text{Gal}(K/\bbq)$ such that $\{\phi_1[\p_1],\phi_1[\p_1],...,\phi_n[\p_m]\}$ is also a zero-sum sequence. We will consider the problem of such norm fusions in more detail later. Let us write $\text{Gal}(K/\bbq) =\{\sigma_i\}_{i=1}^n$ and consider the table

\begin{table}[h]
    \centering
    \begin{tabular}{ccccc}
       $\sigma_1(\p_1)$  & $\sigma_1(\p_2)$  & $\sigma_1(\p_3)$ &  $\dotsm$  & $\sigma_1(\p_m)$  \\
       $\sigma_2(\p_1)$  & $\sigma_2(\p_2)$  & $\sigma_2(\p_3)$ & $\dotsm$  & $\sigma_2(\p_m)$ \\
       $\sigma_3(\p_1)$  & $\sigma_3(\p_2)$  & $\sigma_3(\p_3)$ & $\dotsm$  & $\sigma_3(\p_m)$ \\
        $\vdots$ & $\vdots$ & $\vdots$ & $\ddots$  & $\vdots$ \\
        $\sigma_n(\p_1)$  & $\sigma_n(\p_2)$  & $\sigma_n(\p_3)$ & $\dotsm$  & $\sigma_n(\p_m)$ \\
    \end{tabular}

\end{table}

containing all the factors of $(N(\alpha))$. Once again, we can determine all other elements with norm $N(\alpha)$ (up to a unit multiple) by selecting one prime from each row such that the resulting product is principal. Now, observe that $N(\alpha)$ will be reducible in $N_K$ if and only if there exist $\sigma_{j_k} \in \text{Gal}(K/\bbq)$ such that $\sigma_{j_1}(\p_1)\sigma_{j_2}(\p_2)\dotsm \sigma_{j_m}(\p_m) = (\beta)$ where $\beta$ is reducible in $\mathcal{O}_K$. That is, $\beta = \gamma_1\gamma_2$ where $\gamma_1,\gamma_2$ are nonunits, so $(\beta) = (\gamma_1)(\gamma_2) = \sigma_{j_1}(\p_1)\sigma_{j_2}(\p_2)\dotsm \sigma_{j_n}(\p_m)$. By uniqueness of prime factorization, this corresponds to partitioning $\{\p_1,\p_2,...,\p_m\}$ into two nonempty $\Gamma$-weighted zero-sum sequences. From this, we get the following characterization of $D(\mathfrak{B}_\Gamma(\Cl_K))$.

\begin{prop}
    Let $K$ be a Galois number field and $\Gamma=\text{Gal}(K/\bbq)$. If $\alpha \in \mathcal{O}_K$ with $N(\alpha)$ irreducible in $N_K$, then $D(\mathfrak{B}_\Gamma(\Cl_K))$ is a sharp upper bound for the length of the prime factorization (counting multiplicity) of $(\alpha)$.
\end{prop}

This is especially nice as it mirrors the classic characterization of $D(\Cl_K)$ as a sharp upper bound for the length of the prime factorization of $(\alpha)$ for any irreducible element $\alpha \in \mathcal{O}_K$.

\begin{ex}
    Consider $K = \bbq(\sqrt{-14})$ and $\mathcal{O}_K = \bbz[\sqrt{-14]}$ with $\Cl_K \cong \bbz/4\bbz$ generated by $[\p]$ where $\p = (3,1-\sqrt{-14})$. Now, $\{[\p],[\p],[\p],[\p]\}$ is a minimal 0-sequence of $\Cl_K$, so $\p^4=(5-2\sqrt{-14})$ (and $5-2\sqrt{-14}$ is irreducible in $\bbz[\sqrt{-14}]$). Also, $G = \text{Gal}(K/\bbq) = \{\sigma_1,\sigma_2\}$ where $\sigma_1$ is the identity and $\sigma_2$ is complex conjugation. Thus, $\sigma(\p) = \p$ and $\sigma_2(\p) = \bar{\p} = (3,1+\sqrt{-14})$. Now, let us make our selections.

    \begin{table}[h]
    \centering
    \begin{tabular}{cccc}
        $\boxed{\sigma_1(\p)}$ & $\sigma_1(\p)$ & $\boxed{\sigma_1(\p)}$ & $\sigma_1(\p)$ \\
        $\sigma_2(\p)$ & $\boxed{\sigma_2(\p)}$ & $\sigma_2(\p)$ & $\boxed{\sigma_2(\p)}$ \\
    \end{tabular}
\end{table}

Taking the product yields $\sigma_1(\p)\sigma_2(\p)\sigma_1(\p)\sigma_2(\p) = (\p\bar{\p})^2 = (3)^2 = (9)$. Therefore, we must have $N(5 - 2\sqrt{-14}) =  N(9) = N(3^2)= N(3)^2$, and we see in fact $N(5 - 2\sqrt{-14}) = (5 - 2\sqrt{-14})(5 + 2\sqrt{-14}) = 81 = 9^2 = N(3)^2$. So we see that while $5 - 2\sqrt{-14} \in \OK$ was irreducible, $N(5 - 2\sqrt{-14}) \in N_K$ need not be. 
\end{ex}

A significant majority of the work on $D_\Gamma(G)$ and $D(\mathfrak{B}_\Gamma(G))$ has been in the plus-minus case. That is, where the set of weights is taken to be $\{\indicator, -\indicator\}$ where $-\indicator(g) = -g$. For our context, this corresponds to the case when $K$ is a quadratic ring of integers as $\text{Gal}(K/\bbq) = \{\indicator, \sigma\}$ where $\sigma$ represents quadratic conjugation which acts on $\Cl_K$ by sending any element to its inverse. This follows directly from the transitive action of the Galois group on primes lying over rational primes. We will now consider a particularly nice example outside of this family. 

\begin{ex}
    Let $K = \bbq(\zeta_{37})$. Famously, we have $\Cl_K \cong \bbz/37\bbz$. Now consider the rational prime $149$. As $149 \equiv 1$ mod $37$, we must have that $(149)$ splits into $36$ distinct primes in $K$. Remarkably, it can be shown that each of these distinct primes also lies in a distinct class of $\Cl_K$. Hence, $\text{Gal}(K/\bbq)$ must act transitively on all the non-trivial elements of $\Cl_K$. This means for any minimal zero-sum sequence $\{[\p_1],[\p_2],[\p_3],[\p_4]\}$ in $\Cl_K$, there exist $\sigma_2, \sigma _4 \in \text{Gal}(K/\bbq)$ such that $\{[\p_1],\sigma_2 \cdot [\p_2],[\p_3],\sigma_4 \cdot[\p_4]\} = \{[\p_1],[\p_1]^{-1},[\p_3],[\p_3]^{-1}\}$ (note that we cannot have any $[\p_i]$ trivial because the original zero-sum sequence was minimal). This implies $D(\mathfrak{B}_\Gamma(\Cl_K)) \leq 3$. The sequence $\{[\p^{12}],[\p^{12}],[\p^{12}]\}$ where $\p$ is any non-principal prime shows us that we actually have equality: $D(\mathfrak{B}_\Gamma(\Cl_K)) = 3$. Note that in this case $D_\Gamma(\Cl_K) =2 \neq\mon$. Now, Theorems 5.7 and 7.1 in \cite{BMOS22} tell us that $\rho(N_K) = D(\mathfrak{B}_\Gamma(\Cl_K))/2 = \frac{3}{2}$ while it is well-known that $\rho(\mathcal{O}_K) = \frac{D(\Cl_K)}{2} = \frac{37}{2}$.
\end{ex}

This is a shocking disparity, especially considering $N_K$ is half-factorial if and only if $\mathcal{O}_K$ is half-factorial for a Galois number field $K$ (\cite{JNT}). Heuristically, we would expect the disparity between these two values to increase as the size of the induced subgroup of $\text{Aut}(\Cl_K)$ does. In the case of a quadratic number field with class group $\bbz/2n\bbz$ and $n\geq 2$, we have $D(\mathfrak{B}_\Gamma(\Cl_K)) = 1 + n$ (\cite{BMOS22}), so $\frac{\rho(\mathcal{O_K})}{\rho(N_K)}= \frac{2n}{n+1}$ and the order of the induced subgroup is $2$. For $K = \bbq(\zeta_{37})$, we have $\frac{\rho(\mathcal{O_K})}{\rho(N_K)}= \frac{37}{3}$ and the order of the induced subgroup is $36$. These observations motivate the following conjecture. 

\begin{conj}
    Let $p_i$ be the $i^{th}$ irregular prime and $K = \bbq(\zeta_{p_i})$. 
    \[
    \lim_{i\to \infty} \frac{\rho(\mathcal{O}_K)}{\rho(N_K)} = \infty
    \]
\end{conj}



Connections have recently been made between the arithmetic of this action and coding theory in \cite{borello2025geometry}. We conclude by returning to the quadratic case and drawing an interesting connection between the arithmetic of these normsets and a classic problem from computational complexity theory. In particular, we are interested in the non-trivial norm collisions mentioned earlier in this section. Let us formalize this concept as a decision problem which we will refer to as the equal norm problem.

\begin{equalnorm}
    Let $K$ be a Galois number field and a nonzero nonunit $\alpha \in \mathcal{O}_K$. The equal norm problem is to decide if there exists $\beta \in \mathcal{O}_K$ with $N(\alpha)=N(\beta)$ and $(\sigma(\alpha)) \neq (\beta)$ for any conjugate $\sigma(\alpha)$ of $\alpha$. 
\end{equalnorm}

Note that the condition $(\sigma(\alpha)) \neq (\beta)$ means $\beta$ is not a unit multiple of $\sigma(\alpha)$---precluding trivial solutions. We now recall a classic NP-complete decision problem.

\begin{partition}
    Given a nonempty multiset $S$ of positive integers, the partition problem is to decide whether or not $S$ can be partitioned into two subsets $S_1$ and $S_2$ such that the sums of the numbers in $S_1$ and $S_2$ are equal. 
\end{partition}

Our goal now is to show that any instance of the partition problem can be realized as an instance of the equal norm problem. With this goal in mind, we present the following theorem which follows directly as a porism of Theorem 1 in \cite{ankeny1955divisibility}.

\begin{thm}
    For any positive integer $n$, there exists a square free integer $d$ such that $(3) = \p\q$ in $\mathcal{O}_K$ where $K = \bbq(\sqrt{-d})$, and $[\p]$ has order $n$ in $\Cl_K$.
\end{thm}

For the details of how to find such a $d$, see \cite{ankeny1955divisibility}. This theorem is precisely what we need to make the reduction in question. 

\begin{thm}\label{exist}
    Any instance of the partition problem can be realized as an instance of the equal norm problem for some $\alpha$ in a quadratic number field $K$. What is more, we have a deterministic algorithm for finding such an $\alpha$ and $K$.
\end{thm}

\begin{proof}
     Let $S = \{a_1,a_2,...,a_k\}$ be a nonempty multiset of positive integers and write $n = \sum_{i=1}^k a_i$. By Theorem \ref{exist}, there exists a square free integer $d$ such that $(3) = \p\q$ in $\mathcal{O}_K$ where $K = \bbq(\sqrt{-d})$, and $[\p]$ has order $n$ in $\Cl_K$. Now, consider the 0-sequence $\{[\p^{a_1}],[\p^{a_2}],...,[\p^{a_k}]\}$, and let $\q_i$ be a prime representative of $[\p^{a_i}]$ for all $1\leq i\leq k$. Then, $\{[\q_1],[\q_2],...,[\q_k]\}$ is also a 0-sequence. 

    Now, let $(\alpha) = \q_1\q_2\dotsm\q_k$. All other elements (up to unit multiplication) with norm $N(\alpha)$ can be achieved by discerning which $\q_i$ we can change to $\sigma(\q_i) = \bar{\q}_i$ and achieve a principal ideal, where $\sigma$ is complex conjugation. As $\sigma$ acts on $\Cl_K$ by sending elements to their inverses, finding an element $\beta \in \mathcal{O}_K$ with $N(\alpha) = N(\beta)$ and $(\beta) \neq (\alpha)$ or $(\bar{\alpha})$ corresponds to changing some $0 < t < k$ of the $[\p_i]$ to $[\p_i]^{-1}$ such that, without loss of generality, $\{[\p^{a_1}]^{-1},...,[\p^{a_t}]^{-1},[\p^{a_t}],...,[\p^{a_k}]\}$ is also a 0-sequence. Assume for the purpose of contradiction that $(\alpha) = (\beta) $. Then, by uniqueness of prime ideal factorization, we have some $\sigma(\q_i)=\q_j \implies [\q_i]^{-1}=[\q_j]\implies [\p^{a_i}]^{-1} = [p^{a_j}]\implies [\p^{a_i +a_j}] = \text{Prin}(\mathcal{O}_K) = [\p^n]$. Thus, we must have $a_i + a_j = n$, but this contradicts our assumption that $k \geq 3$. Thus, the decision for the equal norm problem for $\alpha \in \mathcal{O}_K$ is yes if and only if (without loss of generality) we can add some nonzero number of weights weights such that $\{[\p^{a_1}]^{-1},...,[\p^{a_t}]^{-1},[\p^{a_t}],...,[\p^{a_k}]\}$ is a 0-sequence. This will hold if and only if $-a_1 -\dotsm - a_t + a_{t+1} + \dotsm a_k \equiv 0$ mod $n$ $\iff -a_1 -\dotsm - a_t + a_{t+1} + \dotsm +a_k = 0 \iff a_{t+1} + \dotsm +a_k = a_1 +\dotsm + a_t$ which we observe holds (without loss of generality) if and only if the decision for the partition problem for $S$ is yes. Finally, the case when $k\leq 1$ is degenerate, so we only need concern ourselves with $k = 2$. As our only options are yes if $S = \{a,a\}$ and no if $S=\{a,b\}$ where $a=b$, we can pick an arbitrary affirmative and negative instance of the equal norm problem to assign in either case. This completes the proof.
    
\end{proof}

It is important to note that this reduction likely cannot be made in polynomial time because finding prime ideals representatives of the $[\p^{a_i}]$ classes is very likely a hard problem without full knowledge of the class group. Additionally, without further assumptions on the elements in our multiset, we also expect the corresponding norm problem to be intractable in many cases. This expectation is corroborated by the concurrent work done in \cite{mamah2025enhanced} where the author used similar techniques to reduce problems on binary quadratic forms to the subset sum problem in the class group of an imaginary quadratic ring of integers. Regardless, this theorem creates an interesting bridge between computational complexity and the arithmetic of Galois actions. The computations for the following example were conducted in SageMath.  

\begin{ex}
    Consider the multiset $S = \{1,1,2,3,5\}$ with $1 + 1+2+3+5 = 12$. By \cite{ankeny1955divisibility}, we want to find a square-free integer $d$ of the form $d = 3^{12} -x^2$ for some integer $x$ with $2|x$ and $0 < x<(2\cdot 3^{11})^\frac{1}{2}$, and we observe that $d:= 3^{12}-2^2= 531437 = 17\cdot 43 \cdot 747$ suffices. Now, let $K = \bbq(\sqrt{-d})$ so that $\mathcal{O}_K = \bbz[\sqrt{-d}]:=\bbz[\omega]$.

    Now, $(3) = \p_1\p_2 = (3,1+\omega)(3,2+\omega)$. We will take $\p_1=(3,1+\omega)$, and we would like to find prime representatives of the classes $\{[\p_1^1],[\p_1^1],[\p_1^2],[\p_1^3],[\p_1^5]\}$. We choose $\q_1 = \p_1$, $\q_2 = (59077,59061 + \omega)$, $\q_3 = (21379,21165+\omega)$, and $\q_5 = (3167,488+\omega)$ so that $\{[\p_1^1],[\p_1^1],[\p_1^2],[\p_1^3],[\p_1^5]\} = \{[\q_1],[\q_1],[\q_2],[\q_3],[\q_5]\}$ with each $\q_i$ prime. We now compute $\q_1\q_1\q_2\q_3\q_5 = (5766821-2272\omega) := (\alpha)$. 
    
    Now, the equal norm problem for $\alpha$ will be equivalent to the partition problem for $S$. Note that solving the equal norm problem for $\alpha$ in this context is equivalent to finding a solution to the Diophantine equation $a^2 + db^2 = N(\alpha)$ distinct from $(a,b) = (\pm 5766821,\pm2272)$ which correspond to unit multiples of $\alpha$ and $\bar{\alpha}$. We expect this problem to be far more tractable in general than the partition problem. Observe that $(4690134, -51335)$ is one such solution and corresponds to the element $\beta := 4690134 - 51335\omega$. Factoring, we find $(\beta) = \bar{\q}_1\q_1\q_2\q_3\bar{\q}_5$ which delineates a solution to the corresponding partition problem when we observe which primes were conjugated. Taking $S_1 = \{1,5\}$ and $S_2 = \{1,2,3\}$, we see that this is in fact a partition of $S$ into to sets with equal sum. For $n = 12$, there are certainly many smaller quadratic class groups in which we could embed this problem. In the case that the class group into which we are embedding is completely known, this may represent a feasible approach to the partition problem for larger $n$. Furthermore, efficient quantum algorithms exist for the computation of the class group and other computations relevant to this reduction---see \cite{biasse2016efficient} for details.

\end{ex}

\bibliography{References}{}
\bibliographystyle{plain}

\end{document}